\documentclass[a4paper,oneside,11pt]{article}
\usepackage[utf8]{inputenc}
\usepackage{lmodern,stackrel}
\usepackage[T1]{fontenc}
\usepackage{verbatim}
\usepackage{textcomp}
\usepackage[english]{babel}
\usepackage[font=sf, labelfont={sf,bf}, margin=1cm]{caption}
\usepackage[pdftex]{hyperref}
\usepackage[pdftex]{color,graphicx}
\usepackage{amssymb}
\usepackage{amsmath}

\renewcommand{\leq}{\leqslant}
\renewcommand{\geq}{\geqslant}

\usepackage{amsmath,amsfonts,amssymb,amsthm,mathrsfs,stackrel}

\def\build#1_#2^#3{\mathrel{
\mathop{\kern 0pt#1}\limits_{#2}^{#3}}}

\newcommand{\Zst}{\mathbb{Z}^3_*}
\theoremstyle{plain}
\newtheorem{theorem}{Theorem}

\newtheorem{lemma}{Lemma}

\theoremstyle{definition}
\newtheorem{definition}{Definition}

\theoremstyle{remark}

\begin{document}
\title{Embedding  binary sequences into Bernoulli site percolation on $\mathbb{Z}^3$}
\author{M.\ R.\ Hil\'ario\footnotemark[1], B.\ N.\ B.\ de Lima\footnotemark[1], P.\ Nolin\footnotemark[2], V.\ Sidoravicius\footnotemark[3]\\
\footnotemark[1] UFMG, \footnotemark[2] ETH Z\"urich, \footnotemark[3] IMPA}
\date{}
\maketitle

%%%-----------------------------------------------------------------------------------------------------------------------

\begin{abstract}
We investigate the problem of embedding infinite binary sequences into Bernoulli
site percolation on $\mathbb{Z}^d$ with parameter $p$, known also as percolation
of words.\ In 1995, I.\ Benjamini and H.\ Kesten  proved that, for $d \geq 10$ and
$p=1/2$,  all sequences can be embedded, almost surely. They conjectured that
the same should hold for $d \geq 3$. In this paper we consider  $d \geq 3$ and
$p \in (p_c(d), 1-p_c(d))$, where $p_c(d)<1/2$ is the critical threshold for site
percolation on  $\mathbb{Z}^d$. We show that there exists an integer $M = M (p)$,
such that, a.s., every binary sequence, for which every run of consecutive {$0$s} or
{$1$s} contains at least $M$ digits, can be embedded.
\end{abstract}

\section{Introduction}

\noindent {\bf 1.1. Statement of the result. }Fix $d \geq 3$,  and consider Bernoulli site 
percolation on $\mathbb{Z}^d$ with parameter $p \in (0,1)$, \emph{i{.}e{.\ }}take 
$(\Omega, \, \mathcal{A}, \, \mathbb{P}_p)$,
where $\Omega = \{0,1\}^{\mathbb{Z}^d}$, $\mathcal{A}$ is the canonical product
$\sigma$-algebra, and $\mathbb{P}_p = \otimes_{v \in \mathbb{Z}^d} P_p^v$, where
$P_p^v({\omega}_v = 1) = p = 1- P_p^v({\omega}_v = 0)$. An element
$\omega \in \Omega = \{0,1\}^{\mathbb{Z}^d}$ is called a percolation configuration.

Let $\xi \in \Xi:=\{0,1\}^{\mathbb{N}}$, where $\mathbb{N} = \{1,2,\ldots\}$.
For a given $v \in \mathbb{Z}^d$ and  $\omega\in\Omega$, we say that \emph{$\xi$
can be embedded in $\omega$ starting from $v$}, if there exists an infinite
nearest-neighbor vertex self-avoiding path $v_0=v,v_1,v_2,\ldots$ such that
$\omega_{v_i} = \xi_i$ for all $i \geq 1$.  We say that \emph{$\xi$ can be embedded
in $\omega$}, if there exists $v = v(\omega) \in \mathbb{Z}^d$ such that $\xi$ can
be embedded in $\omega$ starting from $v$.

\medskip
Given $v \in \mathbb{Z}^d$ and  $\omega \in \Omega$, define
$$
S_v(\omega):=\{\text{$\xi \in \Xi$ : $\xi$ can be embedded in $\omega$ starting from $v$}\}
$$
and
$$
S_\infty(\omega):=\cup_{v \in \mathbb{Z}^d}S_v(\omega),
$$
which are measurable \cite [Prop. 2]{Benjamini_Kesten}.

\medskip
\noindent {Among other results, Benjamini and Kesten proved in \cite{Benjamini_Kesten}
 that for $d \geq 10$ and $p=1/2$,  all sequences can be embedded, almost
surely.} More precisely,
%%%%%%%%%%%%%%%%
%%%%%%%%%%%%%%%%%%%%%%%%  THEOREM 1 THEOREM 1
%%%%%%%%%%%%%%%%
\begin{theorem}
\label{BK10}
(\cite{Benjamini_Kesten}) {Consider} $\mathbb{Z}^d_+$ with all edges oriented in the positive
direction and $p=1/2$. Then for $d\geq 10$,
\begin{equation}
\label{all}
\mathbb{P}_{1/2}\Big(S_\infty=\Xi\Big)=1,
\end{equation}
and for $d\geq 40$,
\begin{equation}
\mathbb{P}_{1/2}\Big( S(v)=\Xi \, {\text{ for some v}}  \Big)=1.
\end{equation}
\end{theorem}
\medskip

\noindent They also conjectured \cite [Open problem 2, p. 1029]{Benjamini_Kesten} 
that {\eqref{all}} should hold for {all} $d \geq 3$.
\medskip

\noindent The main result of this work makes a step forward in the search for an
affirmative answer {to the question of Benjamini and Kesten}.

\noindent For  $\xi=(\xi_1,\xi_2,\dots ) \in \Xi$, we
define
$I(\xi): =\{i \geq 1 : \xi_i\neq\xi_{i+1}\} .$  Set
$i_0 (\xi):=0,$ and $ i_{j+1}(\xi):= \inf  \big( I(\xi) \setminus \{ i_0 (\xi),\dots,i_j (\xi) \} \big)$.
Let
$$r_j (\xi) := i_{j+1}(\xi) - i_j(\xi).$$
This is well-defined {as soon as} $|I(\xi)| > j$. If $I(\xi)= \{i_1(\xi),  ... , i_j (\xi)\},$ 
for some $j < + \infty$, we {set $r_{j+k} (\xi) := + \infty$ for all $k \geq 1$}. When  
$|I(\xi)| < +\infty$,  we say that $\xi$ is \emph{ultimately monochromatic}, 
and denote {$\Xi_{\textrm{um}} := \{ \xi: \; \xi \text{ is ultimately monochromatic}\}$}.

\begin{definition}
 If there exists $M \geq 1$ such that $r_j({\xi}) \geq M$ for all $j \geq 1$, we say that $\xi$
 is \emph{$M$-stretched}. For a given $M$, we denote 
 $\Xi_M := \{ \xi: \xi \; {\rm{is}} \; M\text{-stretched}\}$.
 \end{definition}
% \noindent \textcolor{red}{In particular, note that an $M$-stretched sequence can also be ultimately monochromatic.}
%\smallskip

Let $p_c (d)$ be the critical
threshold for  site percolation on $\mathbb{Z}^d$, which for $d \geq 3$ is strictly
smaller than $1/2$ (see \cite{Campanino_Russo}).  We are ready to state our
main result:
\begin{theorem}
\label{thm:words} Let $d\geq 3$. Consider site percolation on $\mathbb{Z}^d$
with parameter $p\in (p_c (d), 1-p_c(d))$. There exists
$M(p)$ for which:
$$\mathbb{P}_p\big(\Xi_{M(p)} \subset S_\infty \big)=1.$$
\end{theorem}

\medskip

\noindent {\emph{Remark 1.}} Since $p\in (p_c (d), 1-p_c(d))$, it implies that all
ultimately monochromatic sequences can be embedded almost surely, thus
$$\mathbb{P}_p\big((\Xi_{M(p)} \cup \Xi_{\textrm{um}}) \subset S_\infty \big)=1.$$

\noindent {\emph{Remark 2.}} The following construction shows that for site
percolation on $\mathbb{Z}^d$ with parameter $p=1/2$, the value $M=2$ can be
achieved as soon as $p_c(d-1)<1/4$. For that, consider the subset
$\mathbb{Z}^{d-1} \times \{0,1\}$ of $\mathbb{Z}^d$, and call a vertex
$v = (v_1,\ldots,v_{d-1}) \in \mathbb{Z}^{d-1}$ \emph{good} if
$v' := (v_1,\ldots,v_{d-1},0)$ and $v'':=(v_1,\ldots,v_{d-1},1)$ satisfy 
$\omega_{v'} =0$ and $\omega_{v''} = 1$.
Hence, vertices of $\mathbb{Z}^{d-1}$ are good with probability $1/4$
each, independently of each other. If $p_c(d-1) < 1/4$, then there exists
a.s.\ {an infinite self-avoiding path $\gamma$ of good vertices in $\mathbb{Z}^{d-1}$. 
Then, in the subset of $\mathbb{Z}^{d-1}\times\{0,1\}$ with projection $\gamma$ on $\mathbb{Z}^ {d-1}$,
one can embed any $2$-stretched binary sequence}.
Numerical simulations in \cite{Grassberger} suggest that this happens for $d \geq 5$
(they provide $p_c(4) = 0.19688\ldots$).

\medskip

\noindent {\bf{1.2. Comments and conjectures.}}  It seems that Dekking \cite{D} 
was the first to consider the question of wether $S(v)$ is equal to $\Xi$ in 
the context of {percolation on} regular trees. Benjamini and Kesten \cite{Benjamini_Kesten}
investigated  this problem in a general setup under the name
\emph{percolation of words}, and considered the case $p=1/2$. 
For motivation, historical account and some related works, see their paper and 
the references therein. {Besides the question which motivated our present work, 
another question discussed in  \cite{Benjamini_Kesten} was: what happens
for low-dimensional graphs? In particular, an interesting case is when the value $1/2$ 
is the critical parameter for site percolation as, for example, 
on the triangular lattice $\cal{T}$. Since in this case, a.s.\ neither open nor closed
infinite clusters exist, some sequences cannot be embedded,
and therefore $\mathbb{P}_p(S_\infty = \Xi) = 0$.
Thus, one may ask how rich the set of binary sequences which can be 
embedded is}. Even if one cannot embed all the sequences, it 
is possible that $S_\infty$ consists of ``almost all'' sequences in the following 
sense: let {$\nu_\mu = \otimes_{i=1}^{\infty} \nu_\mu^i$} be the
Bernoulli product measure with parameter $\mu$ on the set of binary sequences
$\Xi$, i.e.\ $\nu_\mu^j(\xi_j =1) = \mu, \; j=1,2, \dots $. 
For a rather general class 
of graphs, and in particular on $\mathbb{Z}^d$, for each $\xi \in \Xi$,
\begin{equation}
\rho (\xi) := \mathbb{P}_p ( \xi {\text{ can be embedded from some $v$}} ) = 
0 \; \rm{or} \; 1.
\end{equation}
We will say that $\xi$ \emph{percolates} if $\rho(\xi) = 1$. Moreover, see  
\cite [(1.12)]{Benjamini_Kesten},
\begin{equation}
\nu_\mu ( \{\xi: \rho(\xi) = 1  \}) = 0 \; \rm{or} \; 1.
\end{equation}
In the former (latter) case we say that almost no sequence (almost all 
sequences, respectively) can be embedded. In 
\cite{Kesten_Sidoravicius_Zhang1} it was shown
that in the case of the triangular lattice $\cal{T}$ and $p=1/2$,
almost all sequences can be embedded
regardless of the value $0<\mu <1$.

\medskip

\noindent Returning to our original question of embeddings on 
{$\mathbb{Z}^d$}, observe that the monochromatic sequences 
$\underline{\mathbf{0}}:=(0,0,\dots)$ and
$\underline{\mathbf{1}} := (1,1,\dots)$ are the least likely to percolate,
in the sense that for any $\xi \in \Xi$ and any $v$,
\begin{align}
\label{W}
    \mathbb{P}_p\Big(\text{$\xi$ can be } &\text{embedded in $\omega$ 
    starting from }v\Big) \notag\\
   & \geq \min_{\zeta \in \{\underline{\mathbf{0}}, \underline{\mathbf{1}}\}} 
   \mathbb{P}_p\Big(\zeta \mbox{ can be embedded in }
   \omega \text{ starting from $v$}\Big),
\end{align}
which follows from \cite [Prop.\ 3.1]{Wierman},  see also
\cite [Lemma 2]{deLima}.
Inequality (\ref{W}) immediately implies that on $\mathbb{Z}^d$, $d \geq 3$,  
for $p\in (p_c (d), 1-p_c(d))$, almost all binary
sequences can be embedded almost surely. Though, the  a.s.\ simultaneous 
occurrence of  $\underline{\mathbf{0}}:=(0,0,\dots)$ and 
$\underline{\mathbf{1}} := (1,1,\dots)$
strongly supports the idea that all binary sequences 
can be embedded, it still remains far from being understood and 
settled. Besides Theorem \ref{BK10} in
\cite{Benjamini_Kesten} mentioned above, for $\mathbb{Z}^d, d\geq 10$, the only low-dimensional result  was obtained in \cite{Kesten_Sidoravicius_Zhang2}, 
where it was shown that 
$\mathbb{P}_{p}(S_\infty=\Xi)=1$ if $p\in (p_c (d), 1-p_c(d))$
 for  $\mathbb{Z}^2_{cp}$ -- the close-packed graph of $\mathbb{Z}^2$, 
 that is, the graph obtained by adding to each face of 
 $\mathbb{Z}^2$ the two diagonal edges.
 
\medskip
%%%%%%%%%%%%%%%%%%%%%%%%%%%%%%%%%%%%%%
%%%%%%%%%%%%%%%%% CONJECTURE CONJECTURE CONJECTURE CONJECTURE
%%%%%%%%%%%%%%%%%%%%%%%%%%%%%%%%%%%%%%
\noindent  {\bf Conjecture and open problems.} The following classification conjecture
was stated by two of the authors\footnote{B.N.B.L. and V.S.}:

\medskip

\noindent  Let $\mathcal{G}$ be {an 
infinite graph, with uniformly bounded degree, and $p_c^{\mathcal{G}}$ 
denote its critical threshold for site percolation}.
\begin{align*}
{\text{I. If }} p_c^{\mathcal{G}}&> 1/2, {\text{then}}\\
%%%%%%%%%%% aaaaaaaaaaaaaaaaaaaaaaaaaaaaaaaaaaaaaaaaaaaaaaaaaa
a) \quad &{\text{For $p \in (0, 1 - p_c^{\mathcal{G}}] \cup [p_c^{\mathcal{G}}, 1)$, 
there exists  $0 <  \mu_c(p) < 1$, such  that }}\\
&{\text{for $p \leq 1 - p_c^{\mathcal{G}}$ almost all binary sequences can be 
 embedded  if  }}\\ 
&{\text{$\mu \leq \mu_c (p)$ and almost no of binary sequences can be embedded }}\\
&{\text{if $\mu_c(p) < \mu$. Similar holds for 
$p_c^{\mathcal{G}} \leq p$: almost all sequences can}}\\
&{\text{be embedded if  $\mu_c (p) \leq \mu$, and almost no sequences can be}}\\
&{\text{embedded if $\mu < \mu_c(p).$}}\\
%%%%%%%%%%% bbbbbbbbbbbbbbbbbbbbbbbbbbbbbbbbbbbbbbbbbbbbbbbbbb
b) \quad  &{\text{If  $p \in  (1 - p_c^{\mathcal{G}}, p_c^{\mathcal{G}})$, 
then for any $\mu$, almost no sequences can be}}\\
&{\text{embedded.}}\\
{\text{II. If }} p_c^{\mathcal{G}} &\leq 1/2, {\text{then}} \\
%%%%%%%%%%% aaaaaaaaaaaaaaaaaaaaaaaaaaaaaaaaaaaaaaaaaaaaaaaaaa
a) \quad  &{\text{For $p \in (0, p_c^{\mathcal{G}}) \cup (1 - p_c^{\mathcal{G}}, 1)$, 
there exists 
$0 <  \mu_c(p) < 1$, such  that }}\\
&{\text{almost all binary sequences can be 
 embedded  for  $\mu \leq \mu_c (p)$ if}}\\
&{\text{$p \leq p_c^{\mathcal{G}}$, and for $\mu_c(p) \leq \mu$ if 
$1- p_c^{\mathcal{G}} \leq p$. Respectively, almost no}}\\
&{\text{sequences can embedded if $\mu > \mu_c(p) $ or $\mu < \mu_c(p).$}}\\
b) \quad &{\text{If $p = p_c^{\mathcal{G}}$
or $p = 1 - p_c^{\mathcal{G}}$, then almost all sequences can be embedded}}\\
&{\text{for all $0<\mu<1$.}}\\
c)\quad &{\text{If
$ 1 - p_c^{\mathcal{G}} < p < p_c^{\mathcal{G}}$, then all sequences
can be embedded.}}
\end{align*} 

\noindent Cases $a)$ of I and II are similar. For the case $p_c^{\mathcal{G}} \leq p$  of I $a)$ 
or  $1 -p_c^{\mathcal{G}} \leq p$ of II $a)$, a Peierls' type argument shows 
that $0 < \mu_c$. To obtain $\mu_c < 1$ is more difficult due to the multi-scale
nature of the problem, and requires elaborated tools. It is a corollary of the
main Theorem 1 of \cite{KSV}. Problems of similar nature are treated in \cite{BS} 
and \cite{KLSV}.
\medskip

\noindent \emph{Open problem 1.} Does I $b)$ hold under these general hypotheses on $\mathcal{G}$, or are some additional assumptions required?
\medskip

\noindent \emph{Open problem 2.} Establish II $b)$. In  \cite{Kesten_Sidoravicius_Zhang1},
it was established for the triangular lattice. However, their proof heavily uses geometric
properties of $\cal{T}$. It would be interesting to establish II $b)$ for 
percolation on the Voronoi tessellation, or for site percolation on a ``random'' triangular 
lattice -- the graph obtained from  $\mathbb{Z}^2$ by adding to its faces only 
one, {randomly chosen diagonal (with angle $\pm \pi/4$)}. It is known (\cite{BR} and 
\cite{R} respectively) that for these graphs the site percolation threshold equals $1/2$.  
The methods of \cite{Kesten_Sidoravicius_Zhang1} do not apply for these models.
\medskip

\section{Proof of Theorem \ref{thm:words}}

\subsection{Preliminaries}

From now on, we restrict ourselves to the case where $d=3$, 
but the same proofs apply to higher dimensions, up to 
minor modifications. The main idea of the proof is to construct 
a pair of self-avoiding $0$- and $1$-paths that approach each 
other with a certain regularity in some particular structures, that we call \emph{outlets}.
This pair of paths is represented schematically in Figure \ref{double_path}.

\begin{figure}
\begin{center}
\includegraphics[width=11cm]{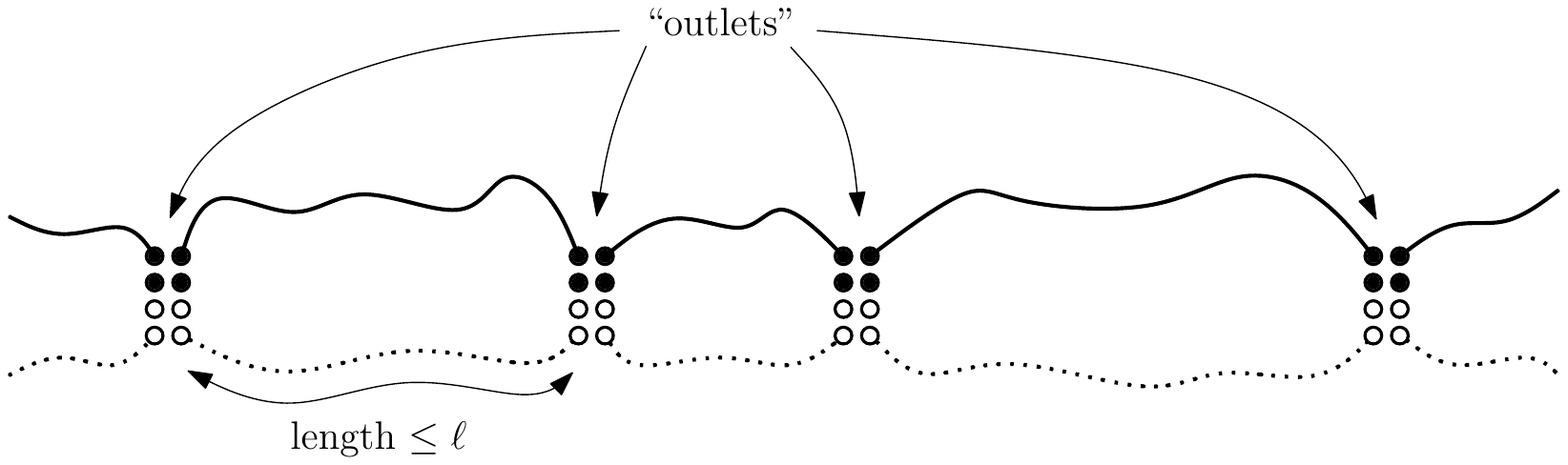}

\caption{\label{double_path} \rm{We construct a double path as 
depicted. The successive ``outlets'' will allow one to switch from a 
$0$-path to a $1$-path, while providing some flexibility at the same 
time. We require all paths connecting successive outlets to have 
length at most $\ell$ (this will allow us to take $M_0 = \ell^2$).}}
\end{center}
\end{figure}

Let us now start the construction.
Considering the box $R = [0,l^{(1)}] \times [0,l^{(2)}] \times [0,l^{(3)}]$, we define
\begin{align}
\mathcal{C}(R) := & \{\text{there exists a 1-path inside $R$, connecting its left side} \nonumber \\
&  [0,l^{(1)}] \times \{0\} \times [0,l^{(3)}] \text{ to its right side }  [0,l^{(1)}] \times \{l^{(2)}\} \times [0,l^{(3)}] \},
\end{align}
and also
\begin{equation}
V^1_t (R) := {U^1_t(R)}^c,
\end{equation}
where
\begin{align}
U^1_t(R) := & \{\text{there exist two {disjoint} connected 1-subsets of $R$, both}  \nonumber \\
& \text{with diameter at least $t$, which are not connected} \nonumber\\
& \text{by a 1-path lying entirely in $R$} \}.
\end{align}
We will also use the notation $V^0_t$ for the event obtained by replacing 1-vertices by 0-vertices in the definition of $V^1_t$.

The following lemma will be used repeatedly.
\begin{lemma}
There exist constants $c_1 = c_1 (p) > 0$ and $c_2 = c_2(p)>0$ such that, for all $N \geq 1$ and all $l^{(j)}_N \in [N, 10 N]$ ($j= 1,2,3$), we have:
\begin{equation} \label{fact1}
\mathbb{P}_p(\mathcal{C}(R)) \geq 1 - c_1 e^{-c_2 N},
\end{equation}
and
\begin{equation} \label{fact2}
\mathbb{P}_p(V^k_N(R)) \geq 1 - c_1 e^{-c_2 N} \quad (k=0,1),
\end{equation}
where $R = [0,l^{(1)}_N] \times [0,l^{(2)}_N] \times [0,l^{(3)}_N]$.
\end{lemma}
These two properties follow from Theorem 5 in \cite{Penrose_Pisztora}. Strictly speaking, they are stated for hypercubes in that paper, but they can be adapted to the case of more general boxes, of the form considered here, by standard gluing arguments.

\subsection{Construction of outlets}

In order to make our arguments more symmetric, we also consider the shifted lattice $\mathbb{Z}^3_{*} = \mathbb{Z}^3 + (0, 1/2, 1/2)$.
We further denote by $\omega \in \{0,1\}^{\mathbb{Z}^3_*}$ a generic $0$ and $1$ site percolation configuration.
We start with a lemma.

Denote by $\Gamma(R_L)$ the event that there exists a 1-path $\gamma$ that stays inside $R_L = [(-L,L) \times (0,8L) \times (0,2L)] \cap \Zst$ and connects $(0, 1/2, 1/2)$ to the right side of $R_L$ (see Figure \ref{one_connection}).

\begin{figure}
\begin{center}
\includegraphics[width=7cm]{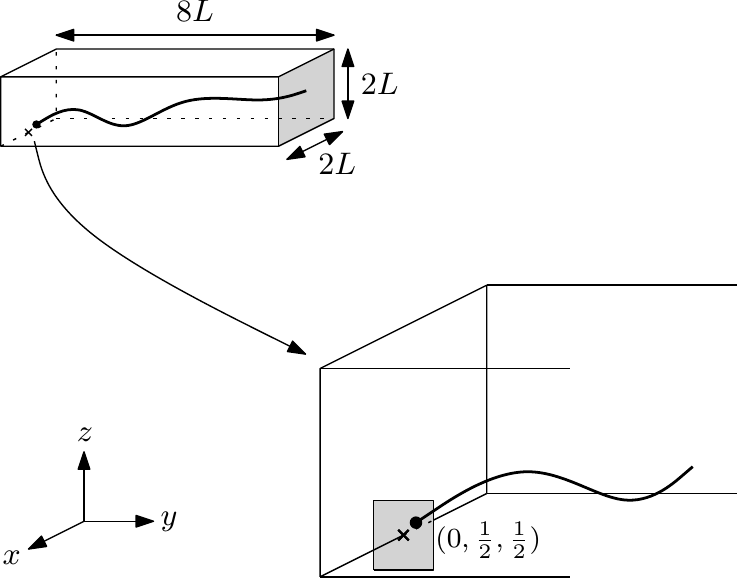}
\caption{\label{one_connection} {\rm{The connecting path from Lemma \ref{lemma:one_connection}.}}}
\end{center}
\end{figure}

\begin{lemma} \label{lemma:one_connection}
 There exist $\delta_0>0$ and $L_0 \geq 1$, depending only on $p$, such that for all $L \geq L_0$,
$$\mathbb{P}_p(\Gamma(R_L)) \geq \delta_0.$$
\end{lemma}

\begin{proof}
We know from \cite{Barsky_Grimmett_Newman} that the critical threshold in a quarter space $\mathbb{Q} := \mathbb{Z} \times \mathbb{Z}_+\times \mathbb{Z}_+$ coincides with $p_c^{\textrm{site}}(\mathbb{Z}^3)$, which implies that with probability at least $\tilde{\delta}_0 = \theta^{\mathbb{Q}}(p)>0$, the vertex $(0, 1/2, 1/2)$ {is connected  to a site $v \in \mathbb{Z}^ 3_*$ with $\|v\|_{\infty} \geq L-1$ by a 1-path $\gamma$ lying inside $R_L$}.
On the other hand, it follows from \eqref{fact1} that there exists a left-right crossing $\tilde{\gamma}$ of $R_L$ with probability at least $1/2$ (for $L$ large enough).
The FKG inequality implies that with probability at least $\tilde{\delta}_0/2$, both paths exist, and these two connected 1-sets can be combined with the help of \eqref{fact2}:  {they are connected by a 1-path staying in $R_L$ with probability arbitrarily close to one (by taking $L$ large enough), which creates a path from $(0, 1/2, 1/2)$ to the right side of $R_L$}.
\end{proof}

Our main construction will be based on ``outlets'', that we define now (see Figure \ref{outlet}):
\begin{definition}
\label{def_outlet}
Given a configuration $\omega \in \{0,1\}^{\mathbb{Z}^3_*}$, we say that the origin $0$ in $\mathbb{Z}^3$ is an \emph{elementary outlet} if $\omega_v=1$  for $v=(0, \pm 1/2, 1/2),(0, \pm 1/2, 3/2)$, and $\omega_v=0$ for $v=(0, \pm 1/2, - 1/2),(0, \pm 1/2, - 3/2)$.
\end{definition}

\begin{figure}[h]
\begin{center}
\includegraphics[width=6cm]{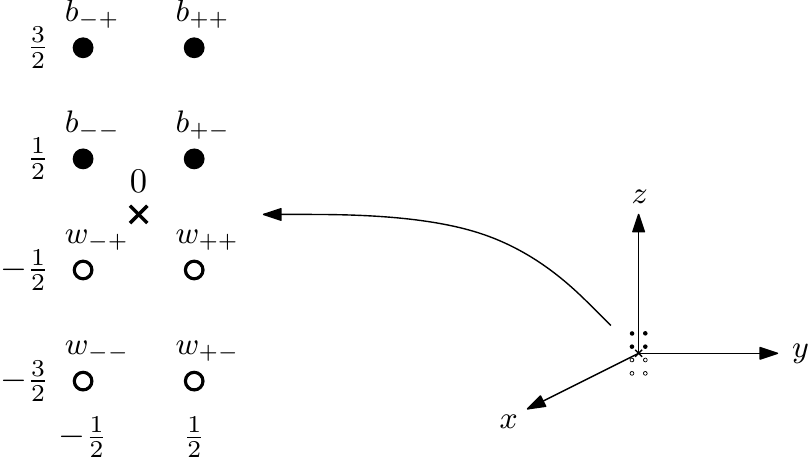}
\caption{\label{outlet} {\rm{An elementary outlet (contained in the plane $x=0$).}}}
\end{center}
\end{figure}

We will later need to refer to the vertices in such an outlet: we denote the 1-vertices by $b_{\pm\pm} = (0,0,1)+(0,\pm 1/2,\pm 1/2)$, and the 0-vertices by $w_{\pm\pm} = (0,0,-1)+(0,\pm 1/2,\pm 1/2)$.

\begin{definition}
\label{def_L_outlet}
Given a configuration $\omega \in \mathbb{Z}^3_*$, we say that the origin $0\in \mathbb{Z}^3$ is an \emph{$L$-outlet} if:
\begin{itemize}
\item[(i)] $0$ is an elementary outlet,

\item[(ii)] and there are four connecting paths $(\gamma_i)_{1 \leq i \leq 4}$, as depicted on Figure~\ref{L_outlet}:
\begin{itemize}
\item $\gamma_1$ (resp. $\gamma_2$) is a 1-path staying inside $R^{(1)}_L = [(-L,L) \times (0,8L) \times (1,2L+1)] \cap \Zst$ (resp. $R^{(2)}_L = [(-L,L) \times (-8L,0) \times (1,2L+1)] \cap \Zst$) and connecting the vertex $(0, 1/2, 3/2)$ (resp. $(0, -1/2, 3/2)$) to the right side of $R^{(1)}_L$ (resp. the left side of $R^{(2)}_L$),
\item $\gamma_3$ (resp. $\gamma_4$) is a 0-path staying inside $R^{(3)}_L = [(-L,L) \times (-8L,0) \times (-2L-1,-1)] \cap \Zst$ (resp. $R^{(4)}_L = [(-L,L) \times (0,8L) \times (-2L-1,-1)] \cap \Zst$) and connecting the vertex $(0, -1/2, -3/2)$ (resp. $(0, 1/2, -3/2)$) to the left side of $R^{(3)}_L$ (resp. the right side of $R^{(4)}_L$).
\end{itemize}
\end{itemize}
{Given a configuration $\omega \in \{0,1\}^{\mathbb{Z}^3_*}$, we say that the vertex $v \in \mathbb{Z}^3$ is an $L$-outlet if $0$ is an $L$-outlet for the configuration $\tau_{v}(\omega)$, where $(\tau_{v}(\omega))_u = \omega_{u+v}$}.
\end{definition}

\begin{figure}
\begin{center}
\includegraphics[width=12cm]{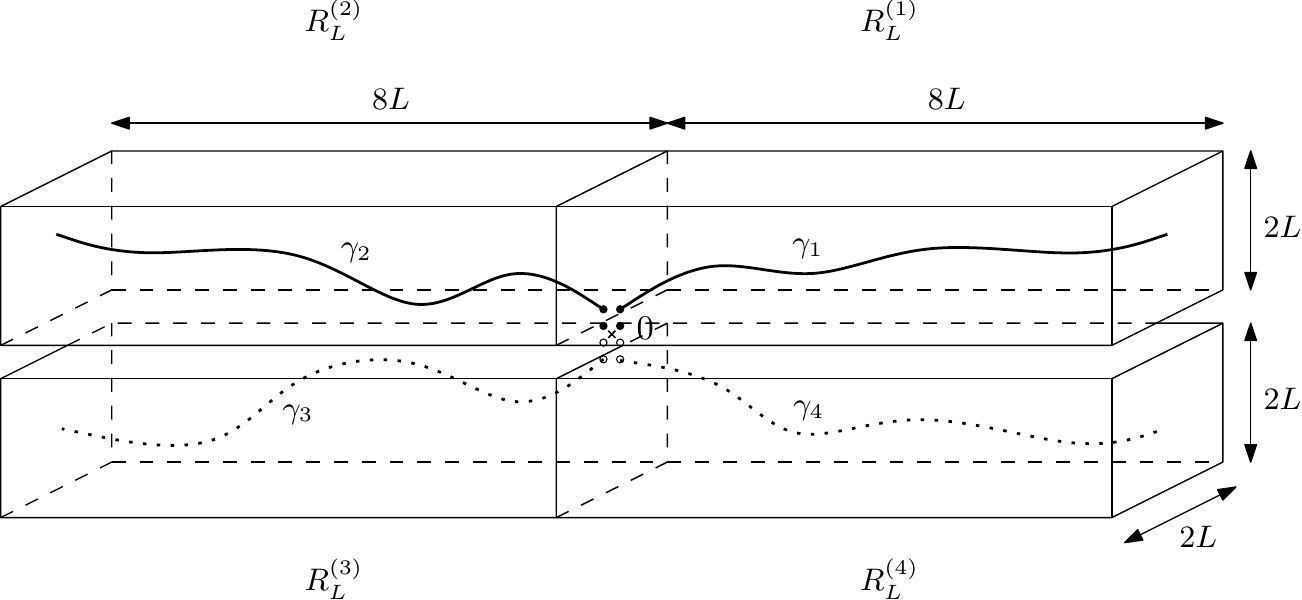}
\caption{\label{L_outlet} {\rm{An $L$-outlet consists of an elementary outlet, together with four connecting paths as depicted.}}}
\end{center}
\end{figure}

\begin{lemma} \label{lemma:0_outlet}
There exist $\delta_1>0$ and $L_1 \geq 1$, depending only on $p$, such that for all $L \geq L_1$,
\begin{equation} \label{proba_outlet}
\mathbb{P}_p(\text{$0$ is an $L$-outlet}) \geq \delta_1.
\end{equation}
\end{lemma}

\begin{proof}
%Recall that $\bar{p}$ was introduced so that $\bar{p} \leq p$ and $\bar{p} \leq 1-p$.
Applying Lemma \ref{lemma:one_connection} with $p$ or $1-p$ to each of the four disjoint boxes $R^{(1)}_L, \ldots, R^{(4)}_L$, we get that with a probability at least $\delta_0(p)^2 \delta_0(1-p)^2$, the four paths $\gamma_1, \ldots, \gamma_4$ all exist. Finally, with probability $p^2(1-p)^2$, each of the four vertices $(0,\pm \frac{1}{2},\pm \frac{1}{2})$ has the right 0-1 value. This completes the proof, with $\delta_1 = p^2(1-p)^2 \delta_0(p)^2 \delta_0(1-p)^2$.
\end{proof}

{\subsection{Block argument}
\begin{definition}
We say that the box
\begin{equation}
B_L = \big((-2L,2L) \times (-8L,8L) \times (-2L-1,2L+1)\big) \cap \Zst
\end{equation}
is \emph{good} if the following three properties are satisfied:
\begin{itemize}
\item[(i)] there exists $k \in (-L,L)$ such that $(k,0,0)$ is an $L$-outlet,

\item[(ii)] the event $V^1_L((-6L,6L)\times (4L,8L) \times (1,2L+1))$ occurs,

\item[(iii)] the event $V^0_L((-6L,6L)\times (4L,8L) \times (-2L-1,-1))$ occurs.
\end{itemize}
\end{definition}

Note that property (i) only depends on the state of the vertices inside $B_L$ (since $B_L$ contains the four boxes in the definition of $(k,0,0)$ being an $L$-outlet).

\begin{lemma} \label{lemma:proba_good}
One has:
\begin{equation}
\mathbb{P}_p (\text{$B_L$ is good})\, {\longrightarrow}\, 1\,  \text{{ as $L \to \infty$}}.
\end{equation}
\end{lemma}
\begin{proof}
{It follows from Lemma \ref{lemma:0_outlet}, and ergodicity of the measure $\mathbb{P}_p$ under lattice translations, that
\[
\mathbb{P}_p\big(\text{$\exists \, k \in (-L,L)$ s.t. $(k,0,0)$ is an $L$-outlet}\big)\, {\longrightarrow}\, 1\, \text{ as } L \to \infty.
\]}
%\[
%\mathbb{P}_p\big(\text{$\exists k \in (-L,L)$ s.t. $(k,0,0)$ is an $L$-outlet}\big) \stackrel{L \to \infty}{\longrightarrow} 1.
%\]
By \eqref{fact2}, we also have that the probabilities of the events $V^1_L((-6L,6L)\times (4L,8L) \times (1,2L+1))$ and $V^0_L((-6L,6L)\times (4L,8L) \times (-2L-1,-1))$ tend to $1$ as $L \to \infty$, so the result follows.
\end{proof}

We now describe the {block argument that will be used in order to prove} Theorem \ref{thm:words}.
For each pair $(i,j) \in \mathbb{Z}^2$, we first introduce $v_L(i,j) = (4iL, 12jL,0)$.
We then define the lattice $\mathbb{Z}^2_L = (\mathbb{V}_L, \mathbb{E}_L)$ having vertex set
\[
\mathbb{V}_L = \{v_L(i,j): \text{$i+j$ is even}\},
\]
and edge set $\mathbb{E}_L$ given by
\[
\langle v_L(i,j), v_L(i',j') \rangle \in \mathbb{E}_L \text{ {if, and only if,}  $\big[|i-i'| = 1$ and $|j-j'|=1\big]$}.
\]
We get in this way an isomorphic copy of $\mathbb{Z}^2$ (see Figure \ref{lattice}). An infinite oriented path in $\mathbb{Z}^2_L$ is a sequence of vertices $v_L(i_0,j_0)$, $v_L(i_1,j_1)$, $v_L(i_2, j_2)\ldots$ such that for all $k \geq 0$, $\langle v_L(i_k,j_k) , v_L(i_{k+1},j_{k+1}) \rangle \in \mathbb{E}_L$, and also $j_{k+1} = j_{k}+1$.
\begin{figure}
\begin{center}
\includegraphics[width=11cm]{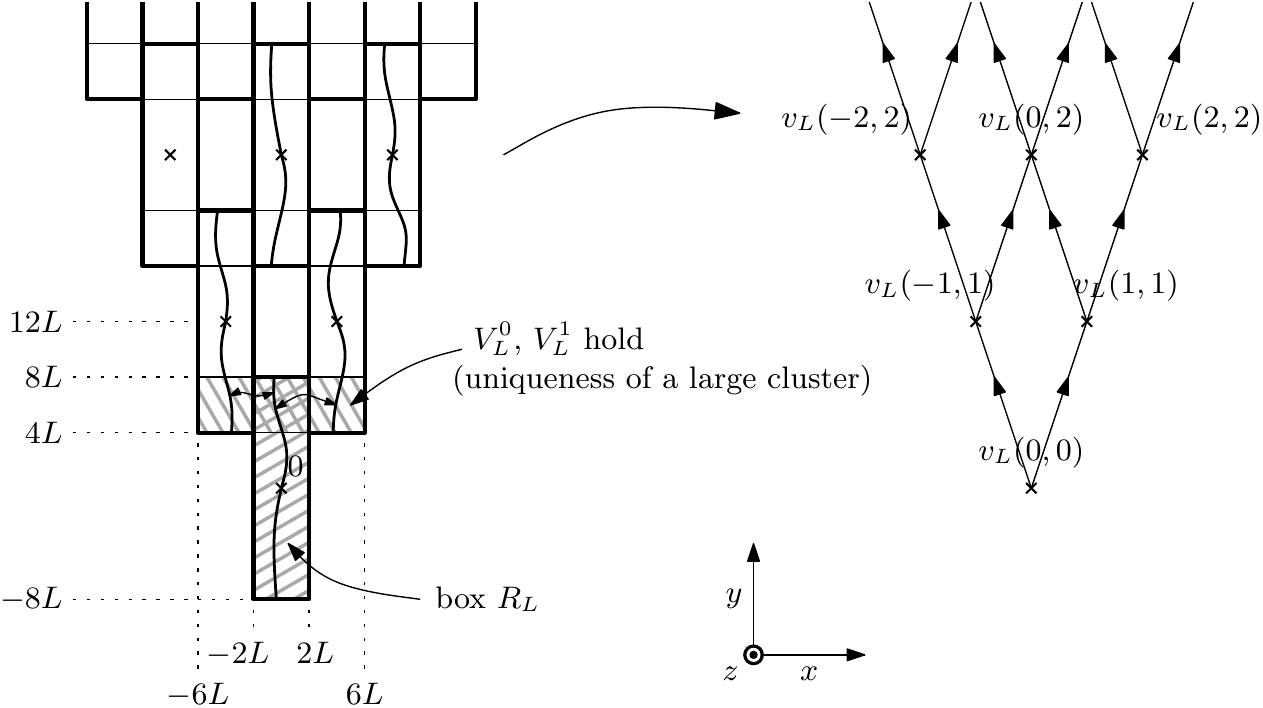}
\caption{\label{lattice} {\rm{The lattice used in the renormalization scheme.}}}
\end{center}
\end{figure}
The vertex $v_L(i,j) \in \mathbb{V}_L$ is said to be \emph{occupied} if the associated box $R_L + v_L(i,j)$ is good.

\begin{lemma}
\label{lemma:oriented}
There exists $L_2$ such that: for all $L \geq L_2$, there exists almost surely an infinite oriented path in $\mathbb{Z}^2_L$ of occupied vertices.
\end{lemma}

\begin{proof}
We know from Lemma \ref{lemma:proba_good} that for any pair $(i,j)$,
\[ 
\mathbb{P}_p(\text{$v_L(i,j)$ is occupied}) =  \mathbb{P}_p(\text{$v_L(0,0)$ is occupied})\rightarrow 1 \text{ as $L \to \infty$.}
\]
Moreover, the event that $v_L(i,j)$ is occupied depends only on the state of the vertices of $\mathbb{V}_L$ within  a graph distance at most two, so that we have a $2$-dependent percolation process. We can thus use a domination by independent percolation \cite{Liggett_Schonmann_Stacey}, which completes the proof.
\end{proof}

\subsection{End of the proof}

We are now in a position to complete the proof. Consider $L_2$ provided by the previous lemma. First, it is an easy observation that from an oriented path as before, one gets a sequence of elementary outlets, where two successive outlets are connected as on Figure \ref{double_path} by 1- and 0-paths of length at most
$$\ell = \ell(L_2) = (12L_2-1) \times 12L_2 \times (2L_2+1).$$

Let us be a bit more precise. If we denote by $(b^i_{\pm\pm})_{i \geq 0}$ and $(w^i_{\pm\pm})_{i \geq 0}$ the vertices that compose the successive outlets, we have constructed two sequences of paths $(\gamma^i_b)_{i \geq 0}$ and $(\gamma^i_w)_{i \geq 0}$, all disjoint of each other, such that for all $i \geq 0$,
\begin{itemize}
\item $\gamma^i_b$ (resp. $\gamma^i_w$) starts at $b^i_{++}$ (resp. $w^i_{+-}$) and ends at $b^{i+1}_{-+}$ (resp. $w^{i+1}_{--}$),

\item and $\gamma^i_b$ (resp. $\gamma^i_w$) has a number $\lambda^i_b$ (resp. $\lambda^i_w$) of vertices (including the extremities) which is at most $\ell$.
\end{itemize}

Let us now take $M_0 = \ell^2$, and explain how to use these paths to {embed} any $M_0$-stretched {sequence} $\xi$. Recall that we denote by $(l_i^\xi)_{i\geq 1}$ the lengths of the successive runs of $0$s and $1$s. Let us first assume that $\xi$ is not ultimately monochromatic, which means that {$\ell^2 \leq l_i^\xi < \infty$} for all $i \geq 1$. We show by induction that the first $j$ runs can be {embedded} starting from the first outlet, and ending in the $k_j$-th outlet (for some $k_j$), on one of the four center vertices $b^{k_j}_{\pm-}$ or $w^{k_j}_{\pm+}$, that we denote by $v_j$.

We just need to explain how to {embed} the $(j+1)$th run, of length $l_{j+1}^\xi \geq M_0 = \ell^2$. We start to {embed} it at $v'_j$, the vertex in the $k_j$th outlet which is adjacent to $v_j$ and has opposite 0-1 value. {We may assume without any loss of generality}, that this run is a 1-run (so that $\omega_{v_j}$ and $\omega_{v'_j}$ are respectively 0 and 1).

Let us introduce $I = \sup \{ i \geq 1 : \lambda^{k_j}_b + \ldots + \lambda^{k_j+i-1}_b \leq l_{j+1}^\xi-4\}$, then
\begin{itemize}
\item from $\lambda^{k_j}_b + \ldots + \lambda^{k_j+I}_b \geq l_{j+1}^\xi-3$ and the definition of $I$, we have
\begin{equation}
\label{eq:sum_lambda}
l_{j+1}^\xi - \ell-3 \leq \lambda^{k_j}_b + \ldots + \lambda^{k_j+I-1}_b \leq  l_{j+1}^\xi-4,
\end{equation}
\item and using that $\lambda^{k_j}_b + \ldots + \lambda^{k_j+I}_b \leq (I+1) \times \ell$, we get that
\begin{equation}
\label{eq:Igeq}
I  \geq \frac{l_{j+1}^\xi-3}{\ell}-1 \geq \frac{\ell^2-3}{\ell} -1 \geq \ell-2.
\end{equation}
\end{itemize}
Let us now consider the path $\rho_{j+1}$ obtained by starting from $v'_j$, following $\gamma^{k_j}_b, \ldots, \gamma^{k_j+I-1}_b$ (if $v'_j$ is $b^{k_j}_{--}$, then we use $b^{k_j}_{+-}$ before following $\gamma^{k_j}_b$), and ending with one extra vertex at $b^{k_j+I}_{--}$. This path has a length $L \in \lambda^{k_j}_b + \ldots + \lambda^{k_j+I-1}_b + \{2,3\}$, which satisfies
$$l_{j+1}^\xi - \ell - 1 \leq L \leq l_{j+1}^\xi -1$$
(using \eqref{eq:sum_lambda}).
First, we can make sure that $L$ and $l_{j+1}^\xi$ have the same parity: if they have different parity, we add $b^{k_j+I}_{+-}$ to the end of $\rho_{j+1}$. We thus get a path $\rho'_{j+1}$, with a length $L'$ satisfying
$$l_{j+1}^\xi - \ell - 1 \leq L' \leq l_{j+1}^\xi.$$
Now, to reach a path of length $l_{j+1}^\xi$ exactly, we just need to play with the $(I-1)$ intermediary outlets, with indices from $k_j+1$ to $k_j+I-1$. Indeed, each of these outlets allows one to add two 1-vertices to $\rho'_{j+1}$ (by making a detour via the two center vertices), and we need to add at most ($\ell+1$) 1-vertices: this can be done since
$$I-1 \geq \ell-3 \geq \frac{\ell+1}{2},$$
which finally shows that the $(j+1)$th run can be seen, as desired. The result then follows by induction.

Clearly, the case when $\xi$ is ultimately monochromatic can be handled in the same way, except that the procedure ends after a finite number of steps. This completes the proof of Theorem \ref{thm:words}.

\section*{Acknowledgements}

V.S. would like to thank Omer Angel for many fruitful discussions. 
P.N. would like to thank UFMG and IMPA, and V.S. the Forschungsinstitut für Mathematik at ETH, for their hospitality.
The research of V.S. was supported in part by Brazilian CNPq grants 308787/2011-0 and  476756/2012-0 and FAPERJ grant E-26/102.878/2012-BBP.
The research of B.N.B.L. was supported in part by CNPq (grants 302437/2011-8 and 470803/2011-8) and FAPEMIG (Programa Pesquisador Mineiro).
This work was also supported by ESF RGLIS Excellence Network.
This research was initiated when P.N. was affiliated with the Courant Institute (New York University), when it was supported in part by the NSF grants OISE-0730136 and DMS-1007626.

\end{document}